\documentclass[a4paper,12pt]{amsart}
\usepackage{amssymb}
\usepackage{ifthen}
\usepackage[dvips]{graphicx}

\nonstopmode \numberwithin{equation}{section}
\usepackage{amssymb}
\usepackage{ifthen}
\usepackage{graphicx}
\usepackage{amsmath}
\usepackage[T1]{fontenc} 
\usepackage[utf8]{inputenc}
\usepackage[usenames,dvipsnames]{color}
\usepackage{color}
\usepackage[english]{babel}
\usepackage{fancyhdr}
\usepackage{fancybox}
\usepackage{tikz}
\usepackage{wasysym}

\nonstopmode \numberwithin{equation}{section}
\theoremstyle{plain}

\newtheorem{conj}{Conjecture}

\theoremstyle{definition}
\newtheorem{defn}{Definition}[section]

\newtheorem{thm}{Theorem}[section]
\newtheorem{prob}{Problem}[section]
\newtheorem{cor}{Corollary}[section]

\newtheorem{prop}{Proposition}[section]
\newtheorem{rem}{Remark}[section]
\newtheorem{lem}{Lemma}[section]

\newcounter{minutes}
\divide\time by 60
\newcounter{hours}
\multiply\time by 60
\addtocounter{minutes}{-\time}

\newcounter {own}
\def\theown {\thesection       .\arabic{own}}

\newenvironment{pf}[1][]{%
	\vskip 3mm
	\noindent
	\ifthenelse{\equal{#1}{}}%
	{{\slshape Proof. }}%
	{{\slshape #1.} }%
}%
{\qed\bigskip}

\newcounter{alphabet}

\def\be{\begin{equation}}
	\def\ee{\end{equation}}

\newcommand{\bee}{\begin{enumerate}}
	\newcommand{\eee}{\end{enumerate}}

\newcommand{\blem}{\begin{lem}}
	\newcommand{\elem}{\end{lem}}
\newcommand{\bthm}{\begin{thm}}
	\newcommand{\ethm}{\end{thm}}
\newcommand{\bcor}{\begin{cor}}
	\newcommand{\ecor}{\end{cor}}
\newcommand{\beg}{\begin{examp}}
	\newcommand{\eeg}{\end{examp}}
\newcommand{\begs}{\begin{examples}}
	\newcommand{\eegs}{\end{examples}}

\newcommand{\bdefn}{\begin{defn}}
	\newcommand{\edefn}{\end{defn}}

\newcommand{\bprob}{\begin{prob}}
	\newcommand{\eprob}{\end{prob}}
\newcommand{\bei}{\begin{itemize}}
	\newcommand{\eei}{\end{itemize}}

\newcommand{\bcon}{\begin{conj}}
	\newcommand{\econ}{\end{conj}}
\newcommand{\bcons}{\begin{conjs}}
	\newcommand{\econs}{\end{conjs}}
\newcommand{\bprop}{\begin{prop}}
	\newcommand{\eprop}{\end{prop}}
\newcommand{\br}{\begin{rem}}
	\newcommand{\er}{\end{rem}}
\newcommand{\brs}{\begin{rems}}
	\newcommand{\ers}{\end{rems}}
\newcommand{\bo}{\begin{obser}}
	\newcommand{\eo}{\end{obser}}
\newcommand{\bos}{\begin{obsers}}
	\newcommand{\eos}{\end{obsers}}
\newcommand{\bpf}{\begin{pf}}
	\newcommand{\epf}{\end{pf}}
\newcommand{\ba}{\begin{array}}
	\newcommand{\ea}{\end{array}}
\newcommand{\beq}{\begin{eqnarray}}
	\newcommand{\beqq}{\begin{eqnarray*}}
		\newcommand{\eeq}{\end{eqnarray}}
	\newcommand{\eeqq}{\end{eqnarray*}}

\begin{document}
	
   \title{Moduli difference of initial inverse logarithmic coefficients for starlike and convex functions}

	\author{Molla Basir Ahamed$^*$}
	\address{Molla Basir Ahamed, Department of Mathematics, Jadavpur University, Kolkata-700032, West Bengal, India.}
	\email{mbahamed.math@jadavpuruniversity.in}
	
	\author{Partha Pratim Roy}
	\address{Partha Pratim Roy, Department of Mathematics, Jadavpur University, Kolkata-700032, West Bengal, India.}
	\email{pproy.math.rs@jadavpuruniversity.in}
	
	\subjclass[{AMS} Subject Classification:]{Primary 30A10, 30H05, 30C35, Secondary 30C45}
	
	\def\thefootnote{}
	\footnotetext{ {\tiny File:~\jobname.tex,
			printed: \number\year-\number\month-\number\day,
			\thehours.\ifnum\theminutes<10{0}\fi\theminutes }
	} \makeatletter\def\thefootnote{\@arabic\c@footnote}\makeatother
	
\begin{abstract} 
Let $\mathcal{A}$ denote the class of functions $f$ that are analytic in the open unit disk $\mathbb{D}$ and satisfy the normalization conditions $f(0) = 0$ and $f'(0) = 1$. This paper investigates the inverse logarithmic coefficients $\Gamma_n$, which are defined by the expansion $\log(f^{-1}(w)/w) = 2\sum_{n=1}^{\infty} \Gamma_n w^n$. We establish sharp upper and lower bounds for the difference of the moduli of the first two inverse logarithmic coefficients, $|\Gamma_2| - |\Gamma_1|$, for several significant subclasses of univalent functions. Specifically, we derive sharp estimates for functions belonging to the class of starlike functions with respect to symmetric points ($\mathcal{S}_S^*$), convex functions with respect to symmetric points ($\mathcal{K}_S$), and functions associated with the lune domain ($\mathcal{S}_{\leftmoon}^*$ and $\mathcal{C}_{\leftmoon}$). The results are obtained by employing subordination techniques and utilizing sharp estimates for the coefficients of Schwarz functions. In each case, the extremal functions that attain these bounds are explicitly identified. Our findings provide further insights into the geometric properties of inverse mappings and extend recent research on coefficient functionals in geometric function theory.
\end{abstract}
	
	\maketitle
	\pagestyle{myheadings}
	\markboth{M. B. Ahamed and P. P. Roy}{\title{On the Moduli Difference of Inverse Logarithmic Coefficients}
	}
	
	\section{\bf Introduction}
	Let $\mathcal{H}$ be the class of functions $f$ which are holomorphic in the open unit disk
	$\mathbb{D}=\{z\in\mathbb{C}:\,|z|<1\}$ of the form
	\begin{align}\label{Eq-1.1}
		f(z)=\sum_{n=1}^{\infty} a_n z^n,\quad \text{for } z\in\mathbb{D}.	
	\end{align}
	
	Then $\mathcal{H}$ is a locally convex topological vector space endowed with the
	topology of uniform convergence on compact subsets of $\mathbb{D}$. Let
	$\mathcal{A}$ be the class of functions $f \in \mathcal{H}$ with $f(0)=0$ and
	$f'(0)=1$, so that
	\begin{align}\label{Eq-1.2}
		f(z)=z+\sum_{n=2}^{\infty} a_n z^n, \quad z\in\mathbb{D}.
	\end{align}

	Let $S$ denote the subclass of all functions in $\mathcal{A}$ which are univalent. For a general theory of univalent functions, we refer the classical books \cite{Duren-1983-NY}. A function $f$ is said to be starlike in $\mathbb{C}$ if it maps the unit disk
	$\mathbb{D}=\{z\in\mathbb{C}:|z|<1\}$ conformally onto a starlike domain with	respect to the origin. The class of functions that are starlike with respect	to symmetric points, denoted by $\mathcal{S}_S^*$, was introduced by Sakaguchi. A function $f\in \mathcal{S}_S^*$ if, and only if,
	\begin{align*}
			{\rm Re}\left(\frac{z f'(z)}{f(z)-f(-z)}\right)>0, \qquad z\in\mathbb{D}.
	\end{align*}
	A function $f$ is said to be \emph{convex} in the complex plane if it maps the unit disk $\mathbb{D}$ conformally onto a convex region. The family of functions that	are convex with respect to symmetric points is denoted by $\mathcal{K}_{S}$ and can be
	characterized by the condition
	\begin{align*}
		{\rm Re}\left(\frac{(z f'(z))'}{f(z)-f(-z)}\right) > 0, \qquad z \in \mathbb{D}.
	\end{align*}
	In geometric function theory, analytic function classes such as starlike and convex	functions are fundamental in describing the geometric behavior of complex mappings. In particular, starlike and convex functions associated with lune domains constitute
	an important area of study. These classes play a key role in the theory of conformal	mappings, which preserve angles locally, and find applications in potential theory, fluid dynamics, and electrostatics. Furthermore, they offer significant insight into
	the geometry of image domains and the structural properties of holomorphic functions.\vspace{1.2mm}

	A function $f \in \mathcal{A}$ is considered starlike (see \cite[p.~40]{Duren-1983-NY}) if it is univalent in $\mathbb{D}$ and its image $f(\mathbb{D})$ is a starlike domain with respect to the origin. Analytically, a function $f$ is called starlike (see \cite[p.~41, Theorem~2.10]{Duren-1983-NY}) if $f\in\mathcal{A}$ and 
	\[
	{\rm Re}\left(\frac{z f'(z)}{f(z)}\right) > 0 \quad \text{for } z \in \mathbb{D}.
	\]
	The class of starlike functions is denoted by $\mathcal{S}^*$.\vspace{1.2mm}
	
	On the other hand, a function $f \in \mathcal{S}$ is called convex (see \cite[p.~40]{Duren-1983-NY}) if $f(\mathbb{D})$ is convex, and the class of convex functions is denoted by $\mathcal{C}$. Moreover, $f \in \mathcal{C}$ if, and only if,
	\begin{align*}
		{\rm Re}\left(z f'(z)\right) \in \mathcal{S}^* \quad \text{for } z \in \mathbb{D}.
	\end{align*}
	
	 The paper shifts focus from the traditional power series coefficients ($a_n$) and standard logarithmic coefficients ($\gamma_n$) to inverse logarithmic coefficients. While $a_n$ and $\gamma_n$ have been studied for decades—most notably in the proof of the Bieberbach conjecture—the study of $\Gamma_n$ for specific subclasses of univalent functions represents a relatively fresh area of research in geometric function theory. The core novelty lies in establishing sharp upper and lower bounds for the difference of the moduli of the first two inverse logarithmic coefficients. Rather than looking at a single coefficient's magnitude, this work examines the `moduli difference' functional, providing deeper insights into the growth and structural properties of the inverse mappings. The work explores subclasses defined by symmetric points, namely $\mathcal{S}^*_S$ (starlike with respect to symmetric points) and $\mathcal{K}_S$ (convex with respect to symmetric points). These classes, introduced by Sakaguchi, are more restrictive than standard starlike or convex classes, and the paper derives sharp results specifically tailored to these symmetric geometric constraints. A significant novel aspect is the association of these coefficients with functions related to the Lune Domain, denoted as $\mathcal{S}^*_{\leftmoon}$ and $\mathcal{C}_{\leftmoon}$. By requiring the quantity $zf'(z)/f(z)$ to lie within a region bounded by a lune (defined by the inequality $|w^2 - 1| \le 2|w|$), the paper connects analytic coefficient theory with modern, non-disk image domains. In ouw work, we utilize the principle of subordination alongside specific technical lemmas, such as Lemma 2.1, to transform geometric conditions into manageable coefficient problems. By employing sharp estimates for Schwarz functions, we provide a rigorous framework for solving functional difference problems that were previously unaddressed for these specific subclasses.\vspace{1.2mm}

	 The principle of subordination is a fundamental and versatile tool in geometric function theory, serving as a cornerstone for defining function classes and solving complex extremal problems. We recall the definition here.
	\begin{defn}(see \cite[p.~43, Theorem~2.12]{Duren-1983-NY})\label{Def-1.1}.
			Let $f$ and $g$ be two analytic functions in $\mathbb{D}$. Then $f$ is said to be subordinate to $g$, written as $f \prec g$ or $f(z) \prec g(z)$, if there exists a function $\omega$, analytic in $\mathbb{D}$ with $\omega(0)=0$, $|\omega(z)|<1$, and
		\[
		f(z)=g(\omega(z)) \quad \text{for } z\in\mathbb{D}.
		\]
		Moreover, if $g$ is univalent in $\mathbb{D}$ and $f(0)=g(0)$, then $f(\mathbb{D}) \subset g(\mathbb{D}).$
	\end{defn}
	In \cite{Raina-Sokol-CRAS-2015}, Raina and Sokol introduced the class $\mathcal{S}^*_{\leftmoon}$, defined by
	\begin{align*}
		\mathcal{S}^*_{\leftmoon}
		:=
		\left\{
		f \in \mathcal{S} :
		\left|\left(\frac{z f'(z)}{f(z)}\right)^2 - 1\right|
		\le
		2 \left|\frac{z f'(z)}{f(z)}\right|,
		\quad z \in \mathbb{D}
		\right\}.
	\end{align*}
	Geometrically, a function $f \in \mathcal{S}^*_{\leftmoon}$ is such that, for any $z \in \mathbb{D}$, the quantity $\frac{z f'(z)}{f(z)}$ lies in the region bounded by a lune. This region is given by
	\[
	\Omega=\left\{ w \in \mathbb{C} : |w^2 - 1| \le 2|w| \right\}.
	\]
	\begin{figure}[htbp]
		\centering
		\includegraphics[width=0.7\textwidth]{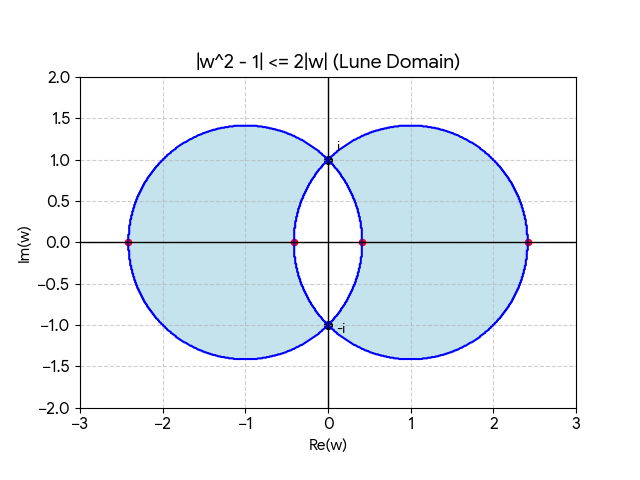}
		\caption{ The figure focuses strictly on the Lune domain $\mathcal{L}$ defined by the inequality $|w^2 - 1| \le 2|w|$. The region is bounded by two symmetric arcs that meet at the vertical vertices $i$ and $-i$. On the real axis, the right-hand lobe extends from $1 - \sqrt{2} \approx -0.414$ to $1 + \sqrt{2} \approx 2.414$. Due to symmetry, the left-hand lobe spans from $-(1 + \sqrt{2})$ to $-(1 - \sqrt{2})$. This region is the range of the function $q(z) = z + \sqrt{1+z^2}$ for $z$ in the unit disk $\mathbb{D}$. In geometric function theory, a function $f$ belongs to the class $\mathcal{S}_{\leftmoon}^{*}$ if the quantity $\frac{zf'(z)}{f(z)}$ takes values within this domain.}
		\label{fig:lune_domain}
	\end{figure}
	By using the definition of subordination, the class $\mathcal{S}^*_{\leftmoon}$ can be equivalently defined as
	\[
	\mathcal{S}^*_{\leftmoon}
	=
	\left\{
	f \in \mathcal{S} :
	\frac{z f'(z)}{f(z)} \prec z + \sqrt{1+z^2} = q(z),
	\quad z \in \mathbb{D}
	\right\},
	\]
	where the branch of the square root is chosen so that $q(0)=1$.\vspace{1.2mm}
		
	Similarly, the class $\mathcal{C}^*_{\leftmoon}$ of convex functions associated with the lune is defined by
	\[
	\mathcal{C}_{\leftmoon}
	:=
	\left\{
	f \in \mathcal{S} :
	1 + \frac{z f''(z)}{f'(z)} \prec q(z),
	\quad z \in \mathbb{D}
	\right\}.
	\]
	
	The class $\mathcal{S}^*_{\leftmoon}$ has been extensively investigated by several authors \cite{Raina-Sokol-CRAS-2015, Raina-Sokol-HJMS-2015, Raina-Sokol-MMN-2015}. The coefficient bounds and the sharp Fekete--Szeg\"o inequality for this class were established by Raina and Sokol (see \cite{Raina-Sokol-HJMS-2015,Raina-Sokol-MMN-2015}). Sharma \textit{et al.}~\cite{Sharma-Raina-Sokol-AMP-2019} examined certain differential subordinations related to $\mathcal{S}^*_{\leftmoon}$. In \cite{Riaz-BSM-2023}, Riaz and Raza established results concerning the third Hankel determinant for starlike and convex functions associated with lune. Integral representations and sufficient conditions for functions in $\mathcal{S}^*_{\leftmoon}$ were provided by Raina \textit{et al.}~(see \cite{Raina-Sokol-CRAS-2015}).\vspace{1.2mm}

	In ${1985}$, de Branges \cite{deBranges-1985} solved the famous Bieberbach conjecture, by showing that if $f \in \mathcal{S}$ of the form \eqref{Eq-1.2}, then $|a_n| \le n$ for $n \ge 2$ with equality	holds for the Koebe function $k(z):=z/(1-z)^2$ or its rotations. It was therefore natural to ask if for $f \in \mathcal{S}$, the inequality $||a_{n+1}|-|a_n|| \le 1$ is true when $n \ge 2$. This problem was first studied by Goluzin \cite{Goluzin-MatSb-1946} with an aim to solve the Bieberbach conjecture. In $1963$, Hayman \cite{Hayman-JLMS-1963} proved that
	$||a_{n+1}|-|a_n|| \le A$ for $f \in \mathcal{S}$, where $A \ge 1$ is an absolute constant and the best known estimate as of now is $3.61$ due to Grinspun \cite{Grinspan-SibInstMat-1976}. On the other hand,
	for the class $\mathcal{S}$, the sharp bound is known only for $n=2$ (see \cite{Duren-1983-NY}, Theorem 3.11),
	namely
	\begin{align*}
		-1 \le |a_3| - |a_2| \le 1.029\ldots
	\end{align*}
	Similarly, for functions $f \in \mathcal{S}^{*}$, Pommerenke \cite{Pommerenke-JDMV-1971} has conjectured that
	$||a_{n+1}|-|a_n|| \le 1$ which was proved later in 1978 by Leung \cite{Leung-BLMS-1978}. For convex
	functions, Li and Sugawa \cite{Li-Sugawa-CMFT-2017} investigated the sharp upper bound of $|a_{n+1}|-|a_n|$ for $n \ge 2$, and sharp lower bounds for $n=2,3$.\vspace{1.2mm}
	
	Let $f \in \mathcal{S}$ and let $F=f^{-1}$ be its inverse, which has the series
	representation
	\begin{align*}
		F(w)=w+\sum_{n=2}^{\infty}A_n w^n,
	\end{align*}
	valid in a neighborhood of the origin. Using the identity
	$f(f^{-1}(w))=w$ and comparing coefficients, we obtain
	\begin{align}\label{Eqn-1.3}
	 A_2=-a_2\; \mbox{and}\; A_3=2a_2^2-a_3.
	\end{align}
	Inverse functions of univalent mappings have been widely studied. The inverse functions are studied by several authors in different perspectives (see, for instance, \cite{Ali-Vasudevarao-PIASMS-2017,Sim-Thomas-Symmetry-2020,Thomas-Tuneski-Vasudevarao-Book-2018} and references therein).
	Recently, Sim and Thomas \cite{Sim-Thomas-Symmetry-2020,Sim-Thomas-BAMS-2021} obtained sharp upper and lower bounds	on the difference of the moduli of successive inverse coefficients for subclasses of univalent functions. For other coefficient problems concerning logarithmic coefficients, readers are referred to the articles \cite{Ahamed-Mandal-F-2026,Ahamed-Mandal-UMJ-2026,Man-Roy-Aha-IJS-2024,Mandal-Ahamed-LMJ-2024,Mandal-Ahamed-Zaprawa-MS-2025} and references therein. \vspace{2mm}
	
	In particular, sharp bounds for differences of successive inverse
	coefficients were recently established for several subclasses.
	The logarithmic coefficients $\gamma_n$ of $f$ are defined by
	\begin{align*}
			\log\frac{f(z)}{z}
		=2\sum_{n=1}^{\infty}\gamma_n z^n,\quad z\in\mathbb{D}.
	\end{align*}
	The logarithmic coefficients $\gamma_n$ are fundamental in the theory of	univalent functions, although only a few sharp bounds are known.
	Their relevance to the Bieberbach conjecture was highlighted by Milin
	\cite{Milin-TMM-1977}, who conjectured that for $f \in \mathcal{S}$ and $n \ge 2$,
	\begin{align*}
		\sum_{m=1}^{n}\sum_{k=1}^{m}
		\left(k|\gamma_k|^2-\frac{1}{k}\right)\le 0.
	\end{align*}
	This conjecture was later proved by de Branges, leading to the solution of	the Bieberbach conjecture. For the Koebe function
	$k(z)=z/(1-z)^2$, one has $\gamma_n=1/n$, but the estimate
	$|\gamma_n|\le 1/n$ does not hold in general for functions in $\mathcal{S}$.	Recent studies have therefore focused on logarithmic coefficients in	$\mathcal{S}$ and its subclasses.\vspace{1.2mm}
	
	Inverse logarithmic coefficients were introduced by Ponnusamy \textit{et al.}~\cite{Ponnusamy-Sharma-Wirths-RM-2018}.	They are defined by
	\begin{align}\label{Eqn-1.4}
		\log\frac{f^{-1}(w)}{w}
		=2\sum_{n=1}^{\infty}\Gamma_n w^n,\qquad |w|<\tfrac14.
	\end{align}
	Differentiating \eqref{Eqn-1.4} and using \eqref{Eqn-1.3}, we obtain
	\begin{align}\label{Eq-1.3}
		\Gamma_1=-\frac{1}{2}a_2\; \mbox{and}\;
		\Gamma_2=-\frac{1}{2}a_3+\frac{3}{4}a_2^2.
	\end{align}
	Ponnusamy \textit{et al.}~\cite{Ponnusamy-Sharma-Wirths-RM-2018} later established sharp bounds for the inverse
	logarithmic coefficients. In particular, for $f \in \mathcal{S}$,
	\begin{align*}
		|\Gamma_n|\le \frac{1}{2n}\binom{2n}{n}, \qquad n\in\mathbb{N}.
	\end{align*}
	Recently, Allu and Shaji (see \cite{AlluShaji2025}) obtained sharp lower and upper bounds for the quantity
	$|\Gamma_2| - |\Gamma_1|$ for functions $f$ belonging to the classes
	$\mathcal{S}$, $\mathcal{S}^*$, $\mathcal{C}$, $\mathcal{S}^*_\alpha$,
	$\mathcal{C}_\alpha$, $\mathcal{S}^*(\alpha)$, $\mathcal{C}(\alpha)$,
	$\mathcal{G}(\nu)$, $\mathcal{F}_0(\lambda)$, $\mathcal{S}^*_\gamma(\alpha)$,
	and $\mathcal{C}_\gamma(\alpha)$. In \cite{Raza-Raza-Bulg-2025}, the authors have established results for moduli difference of initial inverse coefficients of Bazilevič functions.\vspace{2mm}
	
	The primary objective of this paper is to derive sharp lower and upper bounds for
	$|\Gamma_2| - |\Gamma_1|$ for functions $f$ belonging to a given subclass of	$\mathcal{A}$.
   \section{\bf Main results}
   We now state our first main result, which provides sharp bound for
   $|\Gamma_2| - |\Gamma_1|$ when $f$ belongs to the class~$\mathcal{S}^{*}_{S}$.
   \begin{thm}\label{Th-1.1}
   		Let $f \in \mathcal{S}^{*}_{S}$ be given by \eqref{Eq-1.1},
   	then the following sharp inequality holds:
   	\begin{align}
   		\,|\Gamma_{2}| - |\,\Gamma_{1}| \le \frac{1}{2}.
   	\end{align}
   	 The inequality is sharp.
   \end{thm}
   To prove our results, we need the following Lemma.
   
   \begin{lem}\cite{Sim-Thomas-Symmetry-2020}\label{Lem-1.1}
   	Let $B_{1}, B_{2},$ and $B_{3}$ be numbers such that $B_{1} > 0$, 
   	$B_{2} \in \mathbb{C}$, and $B_{3} \in \mathbb{R}$. 
   	Let $p \in \mathcal{P}$ be of the form~(2.1). 
   	Define $\Psi_{+}(c_{1}, c_{2})$ and $\Psi_{-}(c_{1}, c_{2})$ by
   	\[
   	\Psi_{+}(c_{1}, c_{2}) = |B_{2}c_{1}^{2} + B_{3}c_{2}| - |B_{1}c_{1}|,
   	\]
   	and
   	\[
   	\Psi_{-}(c_{1}, c_{2}) = -\,\Psi_{+}(c_{1}, c_{2}).
   	\]
   	Then
   	\begin{align}\label{Eq-2.2}
   		\Psi_{+}(c_{1}, c_{2}) \le 
   		\begin{cases}
   			|4B_{2} + 2B_{3}| - 2B_{1}, & \text{if } |2B_{2} + B_{3}| \ge |B_{3}| + B_{1}, \\[6pt]
   			2|B_{3}|, & \text{otherwise.}
   		\end{cases}
   		\tag{2.2}
   	\end{align}
   	\begin{align}\label{Eq-2.3}
   		\Psi_{-}(c_{1}, c_{2}) \le 
   		\begin{cases}
   			2B_{1} - B_{4}, & \text{if } B_{1} \ge B_{4} + 2|B_{3}|, \\[8pt]
   			2B_{1}\sqrt{\dfrac{2|B_{3}|}{B_{4} + 2|B_{3}|}}, & \text{if } B_{1}^{2} \le 2|B_{3}|(B_{4} + 2|B_{3}|), \\[10pt]
   			2|B_{3}| + \dfrac{B_{1}^{2}}{B_{4} + 2|B_{3}|}, & \text{otherwise,}
   		\end{cases}
   		\tag{2.3}
   	\end{align}
   	where $B_{4} = |4B_{2} + 2B_{3}|$. All inequalities in \eqref{Eq-2.2} and \eqref{Eq-2.3} are sharp.
   \end{lem}
   \begin{proof}[\bf Proof of Theorem \ref{Th-1.1}]
   	If $f\in \mathcal{S}_{S}^*$ be given by \eqref{Eq-1.1}, then by the principle of subordination, there exists a Schwarz function
   	$p(z)=1+\sum_{n=1}^{\infty}c_n z^n$ such that
   	\begin{align}
   		\frac{2 z f'(z)}{f(z)-f(-z)}=p(z). 
   	\end{align}
   	By comparing coefficients of powers of $z$ in both sides, we obtain
   	\begin{align}\label{Eq-1.6}
   		a_{2}=\frac{1}{2}c_{1}, \qquad 
   		a_{3}=\frac{1}{2}c_2.
   	\end{align}
   	Using \eqref{Eq-1.3} and \eqref{Eq-1.6}, we obtain
   	\begin{align}\label{Eq-2.4}
   		|\,\Gamma_{2}| - |\,\Gamma_{1}| 
   		= \frac{1}{16}\!\left(|B_{3}c_{2}+B_{2}c_{1}^{2}|-|B_{1}c_{1}|\right)=\frac{1}{48}\,\Psi_{+}(c_1,c_2)
   	\end{align}
   	where $B_{1}=4, B_{2}=3\; \mbox{and}\;B_{3}=-4.$
   	
   	It is easy to see  that the condition $	|2B_2 + B_3|=2 <|B_3| + B_1=8$
   	holds. Therefore, by Lemma \ref{Lem-1.1}, we have $\Psi_{+}(c_1,c_2)\leq 2|B_3|$. Then from \eqref{Eq-2.4} we have
   	       \begin{align*}
   	       	|\Gamma_2| - |\Gamma_1| \le \frac{2|-4|}{16}=\frac{1}{2}.
   	       \end{align*} 
To show the equality, we consider the function
\begin{align*}
	f_{1}(z)=\frac{z}{1-z^{2}}=z+z^{3}+z^{5}+\cdots, \quad z\in\mathbb{D}.
\end{align*}
Moreover, a simple computation yields
\begin{align*}
	\frac{2z f_{1}'(z)}{f_{1}(z)-f_{1}(-z)}
	=
	\frac{1+z}{1-z}
	=:p_{1}(z).
\end{align*}

The function
\begin{align*}
	p_{1}(z)
	=
	\frac{1+z}{1-z}
	=
	1+2z+2z^{2}+2z^{3}+\cdots
\end{align*}
belongs to the Carath\'eodory class \(\mathcal{P}\). Hence, \(f_{1}\in \mathcal{S}_S^*\).

Now a simple computation shows that
\begin{align*}
	|\Gamma_{2}|-|\Gamma_{1}|
	=
	\frac12.
\end{align*}

This shows that the  inequality in Theorem~\ref{Th-1.1} is sharp.

.\vspace{2mm}

   	 \begin{figure}[htbp]
   	 	\centering
   	 	\includegraphics[width=0.35\textwidth]{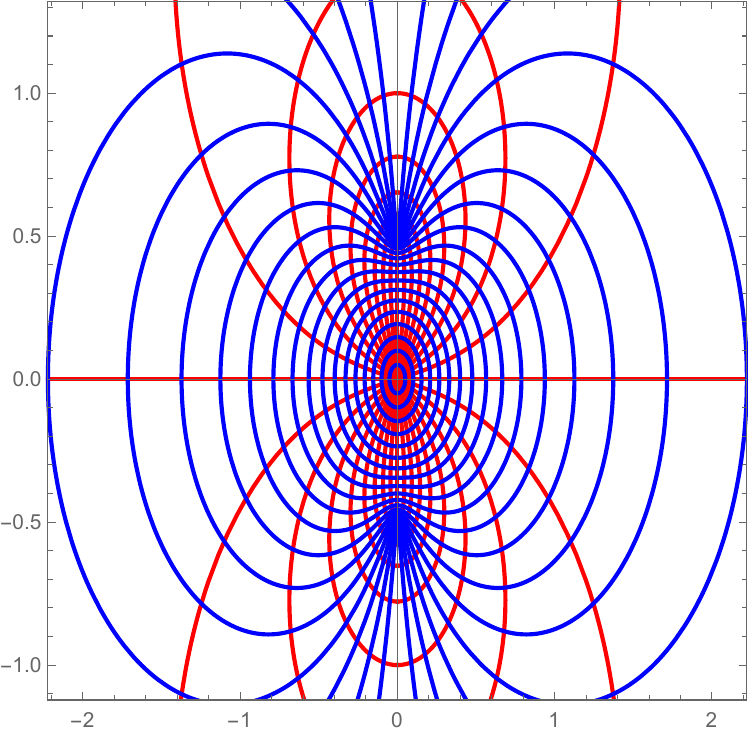}
   	 	\caption{{  The graph exhibits the image domain of $f_1(\mathbb{D})$.	}}
   	 	
   	 \end{figure}
   	
   \end{proof}
   Next, we obtain the upper bound for
   $|\Gamma_2| - |\Gamma_1|$ when $f$ belongs to the class $\mathcal{K}_{S}$.
   \begin{thm}\label{Th-1.2}
   	If $f \in \mathcal{K}_{S}$ be given by \eqref{Eq-1.1},
   	then the following sharp inequality holds:
   	\begin{align}
   	|\,\Gamma_{2}| - |\,\Gamma_{1}| \le \frac{1}{6}.
   	\end{align}
    The inequality is sharp.
   \end{thm}
   \begin{proof}[\bf Proof of Theorem \ref{Th-1.2}]
   	Let $f\in \mathcal{K}_{S}$ be given by \eqref{Eq-1.1}, then by the principle of subordination, there exists a Schwarz function
   	$p(z)=\sum_{n=1}^{\infty}c_n z^n$ such that
   	\begin{align*}
   		\frac{2\bigl(z f'(z)\bigr)'}{(f(z)-f(-z))^{\prime}}=p(z). 
   	\end{align*}
   	By comparing coefficients of powers of $z$ in both sides, we obtain
   	\begin{align}\label{Eq-1.9}
   		a_{2}=\frac{1}{4}c_{1}, \qquad 
   		a_{3}=\frac{1}{6}c_2.
   	\end{align}
   	Now using \eqref{Eq-1.3} together with \eqref{Eq-1.6}, we have
   	\begin{align}\label{Eq-2.7}
   		|\,\Gamma_{2}| - |\,\Gamma_{1}| 
   		= \frac{1}{192}\!\left(|B_{3}c_{2}+B_{2}c_{1}^{2}|-|B_{1}c_{1}|\right)=\frac{1}{16}\,\Psi_{+}(c_1,c_2)
   	\end{align}
   	where $B_{1}=24$, $B_{2}=9$, and $B_{3}=-16$.\vspace{2mm}
   	
   	For the upper bound, we observe that the condition $	|2B_2 + B_3| = 2 < |B_3| + B_1 = 40$ is satisfied. Therefore, by Lemma \eqref{Lem-1.1}, we have $\Psi_{+}(c_1,c_2)\leq 2|B_3|$. Then from \eqref{Eq-2.7} we have
   	\begin{align*}
   		|\Gamma_2| - |\Gamma_1| \le \frac{2|-16|}{192}=\frac{1}{6}.
   	\end{align*}
   	To show the equality, we consider the function
   	\begin{align*}
   		f_{2}(z)=\frac{1}{2}\log\!\left(\frac{1+z}{1-z}\right)z+\frac{z^{3}}{3}+\frac{z^{5}}{5}+\cdots,
   		\;\; z\in\mathbb{D}.
   	\end{align*}
   	Moreover, a simple computation yields
   	\begin{align*}
   		\frac{2\bigl(z f_{2}'(z)\bigr)'}{(f_{2}(z)-f_{2}(-z))'}
   		=
   		\frac{1+z^{2}}{1-z^{2}}
   		=:p_{2}(z).
   	\end{align*}
   	
   	The function
   	\begin{align*}
   		p_{2}(z)
   		=
   		\frac{1+z^{2}}{1-z^{2}}
   		=
   		1+2z^{2}+2z^{4}+\cdots
   	\end{align*}
   	belongs to the Carath\'eodory class \(\mathcal{P}\).
   	Hence, \(f_{2}\in\mathcal{K}_{S}\).

   	A simple computation shows that
   	\begin{align*}
   		|\Gamma_{2}|-|\Gamma_{1}|
   		=
   		\left|-\frac{1}{6}\right|-0
   		=
   		\frac{1}{6}.
   	\end{align*}
   	
   	This shows that the  inequality in Theorem~\ref{Th-1.2} is sharp.

   	\begin{figure}[htbp]
   		\centering
   		\includegraphics[width=0.35\textwidth]{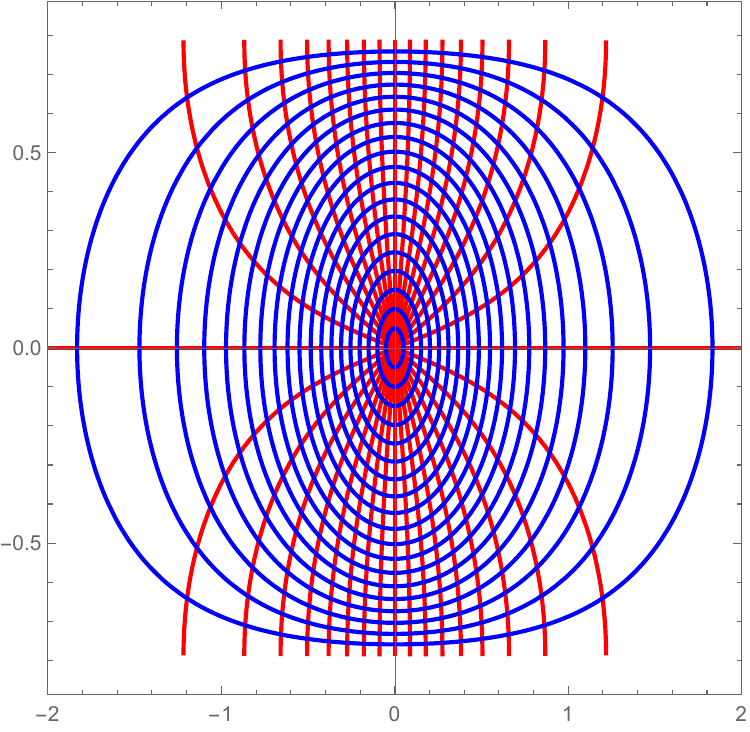}\;\
   		\caption{{  The graph exhibits the image domain of $f_2(\mathbb{D})$.}}
   		
   	\end{figure}
    \vspace{1.2mm}
   
   	 This completes the proof.
   \end{proof}
   We now establish the sharp lower and upper bounds of $ |\Gamma_2| - |\Gamma_1|$
   for the class $\mathcal{S}^{*}_{\leftmoon}$.
   
	 \begin{thm}\label{Th-1.3}
		If $f \in \mathcal{S}^{*}_{\leftmoon}$ be given by \eqref{Eq-1.1},
		then the following sharp inequality holds:
		\begin{align}
			-\frac{1}{\sqrt{10}}\leq|\,\Gamma_{2}| - |\,\Gamma_{1}| \le \frac{1}{4}.
		\end{align}
	 The inequalities are sharp.
	\end{thm}
	\begin{proof}[\bf Proof of Theorem \ref{Th-1.3}]
	    Let $f \in \mathcal{S}^{*}_{\leftmoon}$. Then, in view of Definition \ref{Def-1.1}, it follows that
	    \begin{align}\label{Eq-2.9}
	    	\frac{z f'(z)}{f(z)} = w(z) + \sqrt{1 + w^{2}(z)},
	    \end{align}
	    where $w$ is a Schwarz function with $w(0)=0$ and $|w(z)| \le 1$ in $\mathbb{D}$. 
	    Let $p \in \mathcal{P}$. Then we can write
	    \begin{align}\label{Eq-2.10}
	    	w(z) = \frac{p(z)-1}{p(z)+1}.
	    \end{align}
	    From \eqref{Eq-2.9} and \eqref{Eq-2.10}, a simple computation shows that
	    \begin{align}\label{Eq-2.11}
	    	a_2 = \dfrac{1}{2}\,c_1\; \mbox{and}\;
	    	a_3 = \dfrac{1}{16}\,c_1^{2} + \dfrac{1}{4}\,c_2.
	    \end{align}
		Using \eqref{Eq-1.3} together with \eqref{Eq-2.11}, we see that
		\begin{align}\label{Eq-2.12}
			|\,\Gamma_{2}| - |\,\Gamma_{1}| 
			=\frac{1}{32} \left(|B_{3}c_{2}+B_{2}c_{1}^{2}|-|B_{1}c_{1}|\right):=\frac{1}{32}\Psi_{+}(c_1,c_2)
		\end{align}
		where $B_{1}=8$, $B_{2}=5$, and $B_{3}=-4$.\vspace{2mm}
		
		For the upper bound, we see that the condition $|2B_2 + B_3|=6<|B_3| + B_1=12.$
	 Therefore, by Lemma \ref{Lem-1.1}, we have $\Psi_{+}(c_1,c_2)\leq 2|B_3|$. Then from \eqref{Eq-2.12} we have
		\begin{align*}
			|\Gamma_2| - |\Gamma_1| \le \frac{2|B_3|}{32}=\frac{1}{4}.
		\end{align*} 
	To show the equality, we consider the function
	\begin{align*}
		f_{3}(z)
		&=
		z\exp\!\left(
		\int_{0}^{z}\frac{t^{2}+\sqrt{1+t^{4}}-1}{t}\,dt
		\right)\\&=\dfrac{\sqrt{2}z\exp\left(\frac{z^2-1+\sqrt{1+z^4}}{2}\right)}{\left(\sqrt{1+z^4}+1\right)^{\frac{1}{2}}}\\&=z+\frac{1}{2}z^3+\cdots,
		\quad z\in\mathbb{D}.
	\end{align*}
	
	Moreover, a simple computation yields
	\begin{align*}
		\frac{z f_{3}'(z)}{f_{3}(z)}
		=
		z^{2}+\sqrt{1+z^{4}}
		=:p_{3}(z).
	\end{align*}
	
	The function $	p_{3}(z)
	=
	z^{2}+\sqrt{1+z^{4}}$
	is subordinate to a Carath\'eodory function. Indeed, letting
	\begin{align*}
		p(z)=\frac{1+z^{2}}{1-z^{2}}
		=
		1+2z^{2}+2z^{4}+\cdots \in \mathcal{P},
	\end{align*}
	and using the relation
	$	w(z)=({p(z)-1})/({p(z)+1}),$
	we obtain \(w(z)=z^{2}\). Hence, \(f_{3}\in \mathcal{S}^{*}_{\leftmoon}\).

	Then a simple computation shows that
	\begin{align*}
		|\Gamma_{2}|-|\Gamma_{1}|
		=
		\frac{1}{4}.
	\end{align*}
	
	This shows that the right--hand side inequality in
	Theorem~\ref{Th-1.3} is sharp.
	\begin{figure}[htbp]
		\centering
		\includegraphics[width=0.35\textwidth]{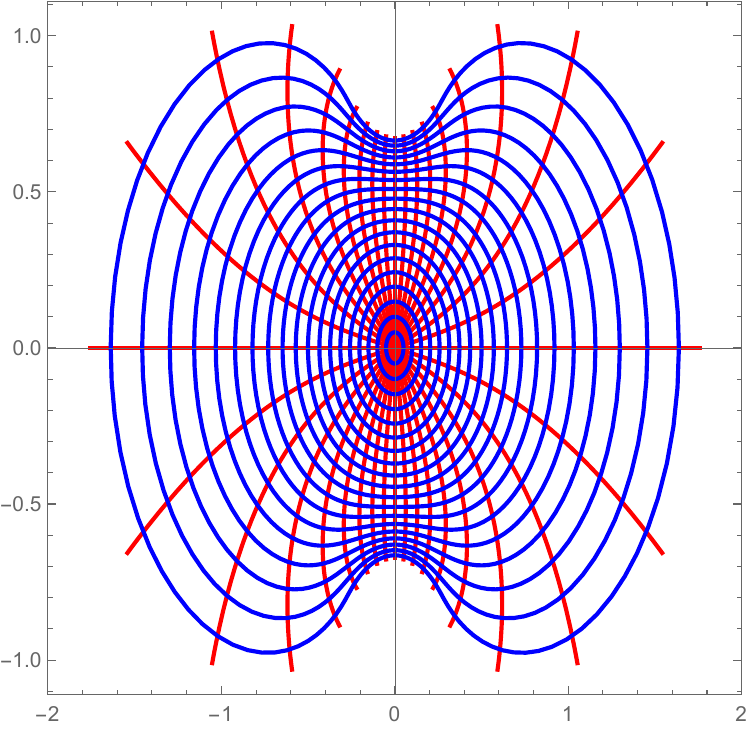}\;\;\;\;\includegraphics[width=0.35\textwidth]{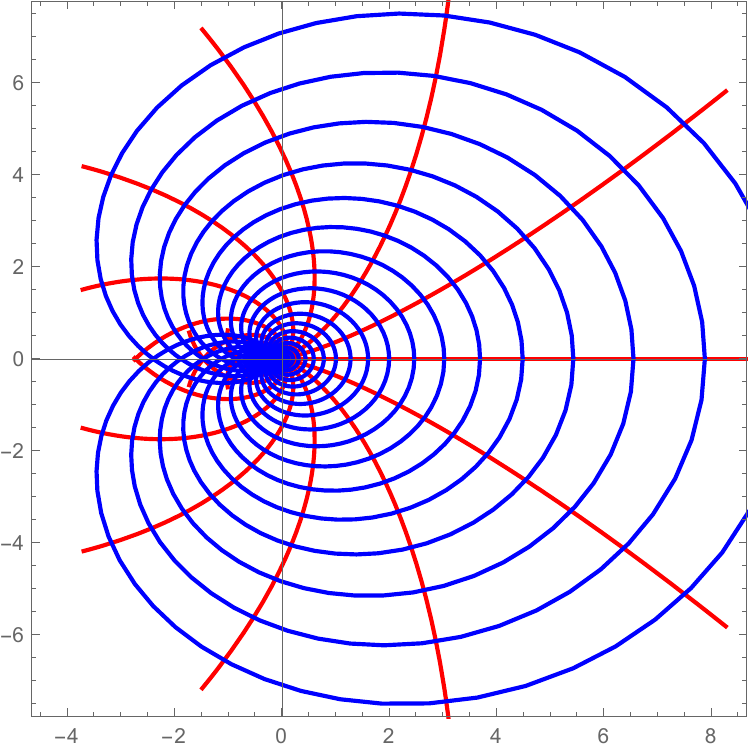}
		\caption{{  The graphs exhibit the image domains of $f_3(\mathbb{D})$ and  $f_4(\mathbb{D})$.}}
		
	\end{figure}
	
		 We now consider the  lower bound. Then
		\begin{align}\label{Eq-2.15}
			|\Gamma_2| - |\Gamma_1|=\frac{1}{32}\Psi_{-}(c_1,c_2),
		\end{align} 
		where $\Psi_{+}(c_1,c_2)=-\Psi_{-}(c_1,c_2)$. Since $B_4=|4B_2+2B_3|=12$, it is easy to see that the inequality $ B_{1} \ge B_{4} + 2|B_{3}|$ does not hold and the inequality
		\begin{align*}
			B_1^2=64< 2|B_3|(B_4+2|B_3|)=160
		\end{align*}
		  holds.
		   Hence, by Lemma \ref{Lem-1.1}, we obtain
		  \begin{align}\label{Eqn-2.16}
		  	\Psi_{-}(c_1,c_2)\leq 2B_{1}\sqrt{\dfrac{2|B_{3}|}{B_{4} + 2|B_{3}|}}=16\sqrt{\frac{2}{5}}.
		  \end{align}
		  Therefore from \eqref{Eq-2.15} and \eqref{Eqn-2.16} we obtain the required inquality 
		  \begin{align*}
		  	|\Gamma_2|-|\Gamma_1|\ge -\frac{1}{\sqrt{10}}.
		  \end{align*}
		 Equality in the left-hand side of Theorem~\ref{Th-1.3} is attained for the function
		 \begin{align*}
		 	p(z)=\frac{1+2Az+z^{2}}{1-z^{2}}, 
		 	\qquad 
		 	A=\frac{2}{\sqrt{10}}.
		 \end{align*}
		 The corresponding extremal function is given by
		 \begin{align*}
		 	f_4(z)
		 	=
		 	z\exp\!\left(
		 	\int_{0}^{z}
		 	\frac{w(t)+\sqrt{1+w^{2}(t)}-1}{t}\,dt
		 	\right),
		 \end{align*}
		 where $	w(z)={Az+z^{2}}/{1-z^{2}}.$
		 This function belongs to the class \(\mathcal{S}^{*}_{\leftmoon}\) and satisfies $|\Gamma_2|-|\Gamma_1|=-{1}/{\sqrt{10}}.$
		 Therefore, the lower bound is sharp.

	\end{proof}
	We now establish the sharp lower and upper estimates of $|\Gamma_{2}-\Gamma_{1}|$
	for the class $\mathcal{C}_{\leftmoon}$.
	
	 \begin{thm}\label{Th-1.4}
		If $f \in \mathcal{C}_{\leftmoon}$ be given by \eqref{Eq-1.1},
		then the following sharp inequality holds:
		\begin{align}
			-\frac{4}{21}\leq|\,\Gamma_{2}| - |\,\Gamma_{1}| \le \frac{1}{12}.
		\end{align}
		 The inequalities are sharp.
	\end{thm}
	\begin{proof}[\bf Proof of Theorem \ref{Th-1.4}]
		Let $f \in \mathcal{C}_{\leftmoon}$. Then, in view of Definition \ref{Def-1.1}, it follows that
		\begin{align}\label{Eq-2.16}
		1 + \frac{z f''(z)}{f'(z)} = w(z) + \sqrt{1 + w^2(z)}.
		\end{align}
		where $w$ is a Schwarz function with $w(0)=0$ and $|w(z)| \le 1$ in $\mathbb{D}$. 
		Let $p \in \mathcal{P}$. Then we can write
		\begin{align}\label{Eq-2.17}
			w(z) = \frac{p(z)-1}{p(z)+1}.
		\end{align}
		From  \eqref{Eq-2.16} and \eqref{Eq-2.17}, a simple computation shows that
		\begin{align}\label{Eq-2.18}
			a_2 = \dfrac{1}{4}\,c_1\; \mbox{and}\;a_3 = \dfrac{1}{48}\,c_1^{2} + \dfrac{1}{12}\,c_2.
		\end{align}
		Using \eqref{Eq-1.3} together with \eqref{Eq-2.18}, we have
		\begin{align}
			|\,\Gamma_{2}| - |\,\Gamma_{1}| 
			= \frac{1}{192}\!\left(|B_{3}c_{2}+B_{2}c_{1}^{2}|-|B_{1}c_{1}|\right)=\frac{1}{16}\,\Psi_{+}(c_1,c_2)
		\end{align}
		where $B_{1}=24$, $B_{2}=7$, and $B_{3}=-8$.\vspace{2mm}
		
		For the upper bound, the condition $	|2B_2 + B_3| =6< |B_3| + B_1 = 32$
		 holds; thus, by Lemma \ref{Lem-1.1}, we obtain
		\begin{align*}
			|\Gamma_2| - |\Gamma_1| \le \frac{2|-8|}{192}=\frac{1}{12}.
		\end{align*} 
	To show the equality, we consider the function
	\begin{align*}
		f_{5}(z)
		&=
		\int_{0}^{z}
		\exp\!\left(
		\int_{0}^{\zeta}
		\frac{t^{2}+\sqrt{1+t^{4}}-1}{t}\,dt
		\right)
		d\zeta\\&=\int_{0}^{z}\dfrac{\sqrt{z}\exp\left(\frac{\zeta^2-1+\sqrt{1+\zeta^4}}{2}\right)}{\left(\sqrt{1+\zeta^4}+1\right)^{\frac{1}{2}}}\\&=z+\frac{1}{6}z^3+\frac{1}{20}z^5+\cdots,\;\; z\in\mathbb{D}.
	\end{align*}
	
	Moreover, a simple computation yields
	\begin{align*}
		1+\frac{z f_{5}''(z)}{f_{5}'(z)}
		=
		z^{2}+\sqrt{1+z^{4}}
		=:p_{4}(z).
	\end{align*}
	
	The function $	p_{4}(z)
	=
	z^{2}+\sqrt{1+z^{4}}$
	is subordinate to a Carath\'eodory function. Indeed, letting
	\[
	p(z)=\frac{1+z^{2}}{1-z^{2}}
	=
	1+2z^{2}+2z^{4}+\cdots \in \mathcal{P},
	\]
	and using the relation
	\[
	w(z)=\frac{p(z)-1}{p(z)+1},
	\]
	we obtain \( w(z)=z^{2} \). Hence,
	\[
	1+\frac{z f_{5}''(z)}{f_{5}'(z)}
	=
	w(z)+\sqrt{1+w^{2}(z)},
	\]
	and therefore \( f_{5}\in \mathcal{C}_{\leftmoon} \).
	
		\begin{figure}[htbp]
		\centering
		\includegraphics[width=0.4\textwidth]{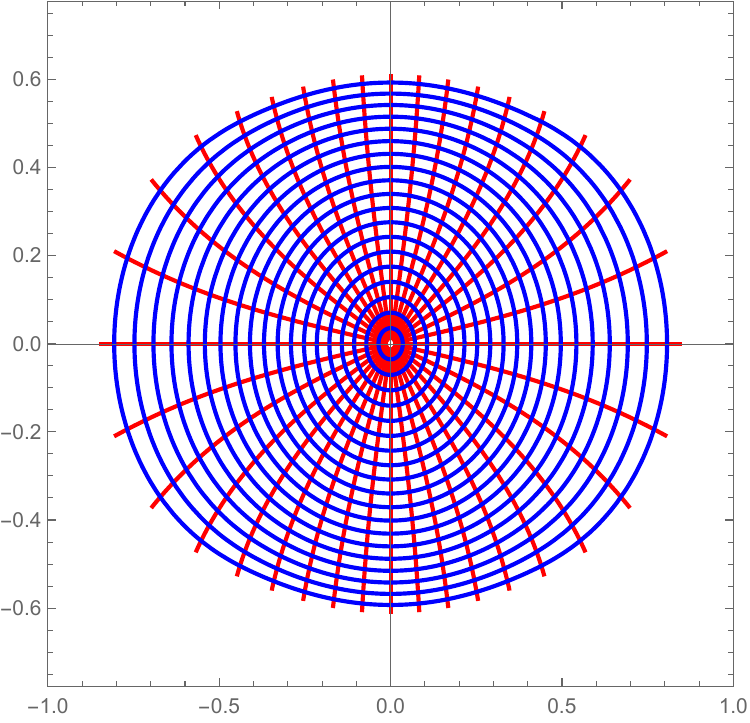}\;\;\;\;\includegraphics[width=0.4\textwidth]{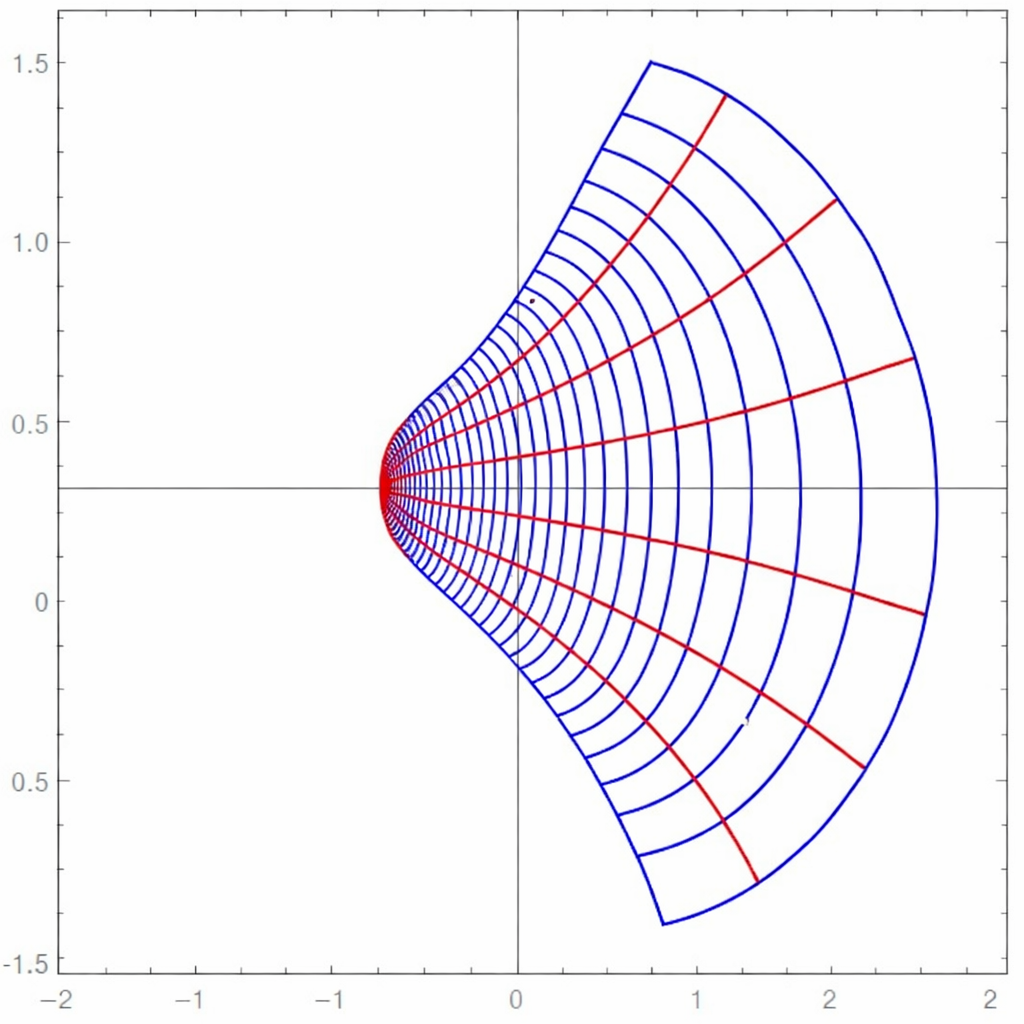}
		\caption{{ The graphs exhibit the image domains of $f_5(\mathbb{D})$ and  $f_6(\mathbb{D})$.}}
	\end{figure}
A simple computation shows that
	\begin{align*}
		|\Gamma_{2}|-|\Gamma_{1}|
		=
		\frac{1}{12}.
	\end{align*}
		 We now consider the  lower bound. Then
	\begin{align}\label{Eq-2.5}
		|\Gamma_2| - |\Gamma_1|=\frac{1}{192}\Psi_{-}(c_1,c_2),
	\end{align} 
	where $\Psi_{+}(c_1,c_2)=-\Psi_{-}(c_1,c_2)$.  Since $B_4 = |4B_2 + 2B_3| = 12,$
	the inequality $B_1 \geq B_4 + 2|B_3|$ is not satisfied. Moreover, $	2|B_3|(B_4 + 2|B_3|) = 448,$
	and hence the inequality $	B_1^{2} \leq 2|B_3|(B_4 + 2|B_3|)$
	is also not satisfied. Hence, by Lemma \ref{Lem-1.1}, we obtain
	\begin{align}\label{Eq-2.6}
		\Psi_{-}(c_1,c_2)\leq \left(2|B_3|+\frac{B_1^{2}}{B_4+2|B_3|}\right)
		= \frac{256}{7}.
	\end{align}
	Therefore from \eqref{Eq-2.5} and \eqref{Eq-2.6} we obtain the required inquality 
	\begin{align*}
		|\Gamma_2|-|\Gamma_1|\ge -\frac{4}{{21}}.
	\end{align*}

		\vspace{2mm}
	The equality in the left--hand side of Theorem~\ref{Th-1.3} for the class
	$\mathcal{C}_{\leftmoon}$ is attained for
	\begin{align*}
			p(z)=\frac{1+2Az+z^{2}}{1-z^{2}}, \qquad A=\frac{4}{7}.
	\end{align*}
	The corresponding extremal function
\begin{align*}
f_6(z)
	=
\int_{0}^{z}
\exp\!\left(
\int_{0}^{\zeta}
\frac{w(t)+\sqrt{1+w(t)^{2}}-1}{t}\,dt
\right)
d\zeta,
\end{align*}
	where
	\begin{align*}
			w(z)=\frac{Az+z^{2}}{1+Az},
	\end{align*}
	belongs to $\mathcal{C}_{\leftmoon}$ and satisfies $|\gamma_{2}|-|\gamma_{1}|=-{4}/{21}.$ Hence, the lower bound is sharp. This completes the proof.
	\end{proof}
		\noindent{\bf Acknowledgment:}  The authors would like thank the anonymous referee for his/her elaborate comments and valuable suggestions which improve significantly the presentation of the paper. The first author is supported by SERB File No.\ SUR/2022/002244, Govt.\ of India, The second author is supported by UGC-JRF (NTA Ref. No.: $ 201610135853 $), New Delhi, India. \vspace{2mm}

		\noindent\textbf{Compliance of Ethical Standards}\\
		
		\noindent\textbf{Conflict of interest} The authors declare that there is no conflict of interest regarding the publication of this paper.\vspace{1.2mm}
		
		\noindent\textbf{Data availability statement}  Data sharing not applicable to this article as no datasets were generated or analyzed during the current study.

\end{document}